\documentclass[11pt]{amsart} 

\usepackage{amsmath,amscd,amssymb,latexsym,color}
\usepackage{longtable}

\allowdisplaybreaks

\newcommand{\sm}{\left(\smallmatrix}
\newcommand{\esm}{\endsmallmatrix\right)}
\newcommand{\mat}{\begin{pmatrix}}
\newcommand{\emat}{\end{pmatrix}}
\renewcommand{\c}{\mathfrak{c}}

\renewcommand{\a}{\alpha}

\newcommand{\G}{\Gamma}

\newcommand{\Q}{\mathbb Q}
\newcommand{\Z}{\mathbb Z}

\newcommand{\R}{\mathbb R}

\renewcommand{\H}{\mathbb H}

\newcommand{\D}{\Delta}

\newtheorem{thm}{Theorem}
\newtheorem{lem}[thm]{Lemma}
\newtheorem{cor}[thm]{Corollary}

\theoremstyle{definition}

\newtheorem{rmk}[thm]{Remark}
\numberwithin{equation}{section}
\numberwithin{thm}{section}

\begin{document}

\title[Weierstrass points]{Weierstrass points at irregular cusps}

\author{Daeyeol Jeon}
\address{Department of Mathematics Education, Kongju National University, Kongju, 314-701, Korea}
\email{dyjeon@kongju.ac.kr}

\begin{abstract}
In this paper, we prove that all irregular cusps on $X_1(N)$ of genus $\geq2$ are Weierstrass points except for $X_1(18)$. 
Also, for any positive integer $N$ of the form $p^2M$ with a prime $p$ and a positive integer $M$, we obtain some results for when the irregular cusps of $X_0(N)$ equivalent to $\sm 1\\p\esm$ are Weierstrass points or not.
\end{abstract}
\maketitle

\renewcommand{\thefootnote}%
             {}
 {}

\section{Introduction}\label{intro}
Let $X$ be a nonsingular algebraic curve of genus $g\geq 2$.
For a given point $P$ on $X$, we say that a positive integer $m$ is a {\it gap} if there is no function on $X$ that has a pole of order $m$ at $P$ and is regular elsewhere. 
The Weierstrass gap theorem says that there are exactly $g$ gaps at $P$, and these gaps $m$ satisfy $1\leq m\leq 2g-1$.
If there is an $m\leq g$ which is a non-gap (also called a {\it jump}),  then we call $P$ a {\it Weierstrass point} of $X$.
It is known that there are only finitely many Weierstrass points on $X$.
Thus except for finitely many Weierstrass points $P$, the gap sequence of $P$ is just $1,2,\dots,g$.
In general, if the gap sequence of $P$ is $a_1,a_2,\dots,a_g$, then the {\it weight} of $P$ is defined by
$${\rm wt}(P)=\sum_{i=1}^{g}(a_i-i).$$
It is known that the total weight $\displaystyle\sum_{P\in X}{\rm wt}(P)=g^3-g$.
In what follows, when we say that $P$ is a Weierstrass point of $X$, then we always assume that $X$ is of genus at least $2$.

Let $\H$ be the complex upper half plane and $\Gamma$ be a congruence subgroup of the full modular group ${\rm SL}_2(\mathbb Z)$.
We consider the modular curve $X(\Gamma)$ obtained from compactification of the quotient space $\Gamma\backslash \H$ by adding finitely many points called {\it cusps}. 
For any integer $N\geq 1,$ we have subgroups $\Gamma(N), \Gamma_1(N),
\Gamma_0(N)$ of ${\rm SL}_2(\mathbb Z)$ defined by matrices $\sm a&b\\c&d\esm$
congruent modulo $N$ to $\sm 1&0\\0&1\esm, \sm 1&*\\0&1\esm, \sm
*&*\\0&*\esm$, respectively. We let $X(N), X_1(N), X_0(N)$ be the
modular curves defined over $\mathbb Q$ associated to $\Gamma(N),
\Gamma_1(N), \Gamma_0(N)$, respectively. 

Let $\Delta$ be a subgroup of $({\mathbb Z}/{N \mathbb Z})^*$, and let
$X_\Delta(N)$ be the modular curve defined over $\mathbb Q$
associated to the congruence subgroup $\Gamma_\Delta(N):$

$$\Gamma_\Delta(N)=\left\{{\begin{pmatrix}a&b\\c&d\end{pmatrix}}
\in\Gamma\,|\,c\equiv 0\,\, {\rm mod}\,\, N, (a\,\, {\rm mod}\,\, N)\in\Delta\right\}.$$

Since the negative of the unit matrix acts as identity on the complex upper halfplane we
have $X_\Delta(N)=X_{\langle \pm1,\Delta\rangle}(N)$; hence we always assume that $-1\in\Delta.$ 
\par 
If $\Delta=\{\pm 1\}$ (resp.
$\Delta=({\mathbb Z}/{N \mathbb Z})^*$) then $X_\Delta(N)$ is
equal to $X_1(N)$ (resp. $X_0(N)$). The inclusions 
$\Gamma(N)\subseteq \pm\Gamma_1(N)\subseteq\Gamma_{\Delta}(N)
\subseteq\Gamma_0(N)$
induce natural Galois coverings 
$X(N)\to X_1(N)\to X_{\Delta}(N)\to X_0(N)$. 

Denote the genera of $X_0(N)$, $X_1(N)$ and $X_\Delta(N)$ by $g_0(N)$, $g_1(N)$ and $g_\Delta(N)$, respectively.
\par

According to Ogg \cite{O1}, the cusps on $X(N)$ can be regarded as pairs $\pm\sm x\\y\esm$ where $x,y\in\Z/N\Z$ are relatively prime and $\sm x\\y\esm$, $\sm -x\\-y\esm$ are identified, and a cusp on $X_\Delta(N)$ can be regarded as an orbit of $\Gamma_\Delta(N)/\Gamma(N)$.

We call a cusp of the form $\sm x\\N\esm$ an {\it $\infty$-cusp}, and a cusp of the form $\sm x\\y\esm$ with $\gcd(y,N)=1$ a {\it $0$-cusp}.
Also, for a cusp $s=\sm x\\y\esm$ with $d=\gcd(y,N)$, if $\gcd(d,N/d)=1$, then $s$ is called a {\it regular cusp}, and it is called an {\it irregular cusp} otherwise. Thus if $N$ is square-free, then all cusps are regular cusps.
For any two cusps $s,t$ of $X_\Delta(N)$, if there is an automorphism $\sigma$ on $X_\Delta(N)$ so that $t=\sigma(s)$, then $s,t$ are said to be ${\it equivalent}$.
Obviously equivalent cusps all behave the same with respect to being Weierstrass or not.

The Weierstrass points of $X(N)$ were studied by Petersson \cite{P} and Schoeneberg \cite{S}, and
Schoeneberg proved that all cusps of $X(N)$ are Weierstrass points for $N\geq 7$.

By using Schoeneberg's result, Lehner and Newman~\cite{LN} gave sufficient conditions for when the $\infty$-cusp is a Weierstrass point of $X_0(4M)$ and $X_0(9M)$.
We note that if the $\infty$-cusps are Weierstrass points, so are the $0$-cusps because the Fricke involution $W_N$ maps the $\infty$-cusps to the $0$-cusps.
Later, Atkin extended their results to the case $X_0(p^2M)$, where $p\geq 5$ is any prime.
On the other hand, Ogg \cite{O3} proved that the $\infty$-cusp, indeed any $\Q$-rational cusp, is not a Weierstrass point on $X_0(pM)$ if a prime $p\nmid M$ and $g_0(M)=0$.
Especially, the $\infty$-cusp is not a Weierstrass point on $X_0(p)$ for any prime $p$.
 
We can see that the tendency of a cusp to be a Weierstrass point is different in $X(N)$, $X_1(N)$ and $X_0(N)$.
For example, when $p$ is a prime, an $\infty$-cusp is always a Weierstrass point of $X(p)$, but never of $X_0(p)$,
whereas on $X_1(p)$ the cusps sometimes are Weierstrass points and sometimes not.

The study on the Weierstrass points of $X_1(N)$ was less done than for the other modular curves.
Especially, it has not been studied when irregular cusps of modular curves are Weierstrass points or not.
The author, Im and Kim \cite{IJK} used Atkin's method \cite{A} to investigate $X_\Delta(N)$, but mainly under the condition $\pm(1+pM)\in\Delta$ for some prime $p$; hence the result can not be applied to $X_1(N)$.
Recently Choi~\cite{Ch} proved that the irregular cusp $\sm 1\\ 2\esm$ is a Weierstrass point on $X_1(4p)$ for any prime $p>7$. 
In this paper, we generalize Choi's result to obtain the following:
\begin{thm}\label{main}
All irregular cusps on $X_1(N)$ are Weierstrass points except for $N=18$.
\end{thm}

We may compare this result with Schoeneberg's result in \cite{S} that all cusps of $X(N)$ are Weierstrass points for $N\geq 7$.

Also, for any positive integer $N$ of the form $p^2M$ with a prime $p$ and a positive integer $M$, we obtain some results for when the irregular cusps of $X_0(N)$ equivalent to $\sm 1\\p\esm$ are Weierstrass points or not.




\section{Preliminaries}\label{pre}

For our purpose, we need to know how to describe the cusps on $X_\D(N)$.
In virtue of \cite{O1} we have the following description of cusps. 
The quotient group $\Gamma_\D(N)/{\Gamma(N)}$ operates naturally on the left, and so a cusp of $X_\Delta(N)$ can be regarded as an orbit of $\Gamma_\D(N)/{\Gamma(N)}.$ 
For each $d\mid N,$ a cusp of $X_1(N)$ is represented by a pair $\sm x\\y\esm$ with $x$ reduced modulo $d=\gcd(y,N)$ and $(x,d)=1.$ 
Also for each $d\mid N,$ a cusp of $X_0(N)$ is represented by a pair $\sm x\\d\esm$ with $x$ reduced modulo $e=\gcd(d,N/d)$. 
From this argument, we have the result:

\begin{lem}\cite{O1}\label{lem:cusp} For each $d\mid N$, let $e=\gcd(d,N/d)$.
\begin{enumerate}
\item Suppose $g_1(N)>0$, i.e. $N=11$ or $N\geq 13$. Then for each $d\mid N$, we have $\frac12\varphi(d)\varphi(N/d)$ cusps $\sm x\\y\esm$ of $X_1(N)$ with $d=(y,N)$.
\item For each $d\mid N$, we have $\varphi(e)$ conjugate cusp $\sm x\\d\esm$ of $X_0(N)$, each with ramification index $e$ in the Galois covering $X_1(N)\to X_0(N)$, and these are all the cusps of $X_0(N)$.
In particular, a cusp is defined over $\Q$ only if $d=1$ or $2$.
\end{enumerate}
\end{lem}

For $d\mid N,$ let $\pi_d$ be the natural projection from $({\mathbb Z}/{N \mathbb Z})^*$ to
$({\mathbb Z}/{{\rm lcm}(d,N/d)\mathbb Z})^*$.
From the Hurwitz formula, the following genus formula is obtained:

\begin{thm}\cite{JK}\label{Gn}
The genus of the modular curve $X_\Delta(N)$ is given by
$$g_\D(N)=1+\frac{\mu(N,\D)}{12}-\frac{\nu_2(N,\D)}{4}-\frac{\nu_3(N,\D)}3-\frac{\nu_\infty(N,\D)}2$$
where
\begin{eqnarray*}
\mu(N,\D) &=& N\cdot\prod_{p|N\atop
prime}\left(1+\frac1p\right)\cdot\frac{\varphi(N)}{|\Delta|}\\
\nu_2(N,\D) &=& \left|\{(b\mod N)\in\Delta\,|\, b^2+1\equiv 0\mod
N\}\right|\cdot\frac{\varphi(N)}{|\Delta|}\\ 
\nu_3(N,\D) &=&
\left|\{(b\mod N)\in\Delta\,|\, b^2-b+1\equiv 0\mod
N\}\right|\cdot\frac{\varphi(N)}{|\Delta|}\\ 
\nu_\infty(N,\D) &=&
\sum_{d|N}\frac{\varphi(d)\varphi(\frac Nd
)}{|\pi_d(\Delta)|}.\end{eqnarray*}
\end{thm}

For an integer $a$ prime to $N,$ let $[a]$ denote the automorphism of $X_\Delta(N)$ represented by $\gamma\in\Gamma_0(N)$ such that $\gamma\equiv\sm a&*\\0&*\esm\mod N.$ 
Sometimes we regard $[a]$ as a matrix.

For each exact divisor $Q\|N$, i.e. $(Q,N/Q)=1,$ consider the matrices of the form $\begin{pmatrix} Qx & y\\Nz & Qw \end{pmatrix}$ with $x,y,z,w\in\mathbb Z$ and determinant $Q.$ 
Then these matrices define a unique involution of $X_0(N)$, which is called the {\it Atkin-Lehner involution} and denoted by $W_Q$. 
In particular, if $Q=N,$ then $W_N$ is called the {\it Fricke involution.}
We also denote by $W_Q$ a matrix of the above form.
$W_Q$ also normalizes $\Gamma_1(N)$, but the automorphism it induces on $X_1(N)$ is not necessarily an involution. 
Also, $W_Q$ does not necessarily normalize $\Gamma_\Delta(N)$.

Kim and Koo \cite{KK} proved that the normalizer $N(\Gamma_1(N))$ of $\Gamma_1(N)$ in $PSL_2(\mathbb R)$ is generated by the matrices $[a]$ and $W_Q$ except for $N=4$.
Any element of $N(\Gamma_1(N))$ defines an automorphism of $X_1(N)$, we can regarded the quotient group $N(\Gamma_1(N))/\pm\Gamma_1(N)$ as a subgroup of the automorphism group ${\rm Aut}(X_1(N))$ of $X_1(N)$.
Indeed, when $N\geq 4$, the automorphism group ${\rm Aut}(Y_1(N))$ is isomorphic to $N(\Gamma_1(N))/\pm\Gamma_1(N)$ where $Y_1(N)$ is the affine modular curve corresponding to $\Gamma_1(N)$.

All regular cusps of $X_0(N)$ are on the same orbit under the automorphism group, actually already under the action of the Atkin-Lehner involutions, and hence all are Weierstrass or none. 
The same holds for the regular cusps of $X_1(N)$ under the action of automorphims $[a]$ and $W_Q$.
However, some $W_Q$ may not normalize $\Gamma_\Delta(N)$, so it can happen that not all regular cusps of $X_\Delta(N)$ are on the same orbit under the automorphism group.


 


The following result is Schoeneberg's Theorem \cite{S} which plays a central role in this paper. 
For the sake of completeness we present the very short proof.

\begin{thm}\label{central} \cite{S} Suppose that $X$ is a non-singular algebraic curve of $g\geq 2$.
Let $\sigma$ be an automorphism of order $m$ with a fixed point $P$, and $\bar{g}$ denote the genus of the quotient curve $X/\langle \sigma\rangle$.
If $g-m\bar{g}\geq m$, then $P$ is a Weierstrass point.
\end{thm}
\begin{proof}
Let $\bar{P}$ be the corresponding point of $P$ on $X/\langle\sigma\rangle$.
On $X/\langle\sigma\rangle$ there is a function with a single pole of order at most $\bar{g}+1$ at $\bar{P}$.
This lifts to a function on $X$ with a single pole at $P$ of order at most $m(\bar{g}+1)$.
 As soon as this number is smaller or equal $g$, $P$ is a Weierstrass point.
\end{proof}



From Theorem \ref{central} and the Hurwitz formula we get the following result which is known as Lewittes' Theorem.
\begin{cor}\label{cor2}
\cite{L} The assumptions and notations are the same as in Theorem \ref{central}.
If $\sigma$ has more than 4 fixed points, then $P$ is a Weierstrass point.
\end{cor}




\section{Weierstrass points of $X_1(N)$ at irregular cusps}\label{sec3}


From now on, we always assume that $N$ is not square-free, $d\mid N$ with $\gcd(d,N/d)>1$, $$e=\gcd(d,N/d),$$ and 
$$\Delta_d=\left\{a\in (\Z/N\Z)^*\,|\, a\equiv\pm 1\pmod{N/e}\right\}.$$
Each $[a]$ acts on $\sm x\\y\esm$ with $d=\gcd(y,N)$ as $\sm ax\\a^{-1}y\esm$; hence $[a]\sm x\\y\esm=\sm x\\y\esm$ on $X_1(N)$ if and only if  $a\equiv\pm1\pmod{d}$ and $a\equiv\pm1\pmod{N/d}$, i.e.  $a\in\D_d$.
Note that all irregular cusps are represented by $\sm x\\ y\esm$ with $d=\gcd(y,N)$ for some $d$.

Concerning the irregular cusps, we have the following result:

\begin{lem}\label{lem:tot} The map $X_1(N)\to X_{\D_d}(N)$ is of degree $e$ and it is totally ramified at each irregular cusp $\sm x\\ y\esm$ with $d=\gcd(y,N)$.
\end{lem}
\begin{proof} Since $[\Gamma_{\D_d}(N):\pm\Gamma_1(N)]=e$, the map $X_1(N)\to X_{\D_d}(N)$ is of degree $e$.
Also for any $a\in\D_d$, $[a]$ fixes $\sm x\\y \esm$ with $d=\gcd(y,N)$; hence the result follows.
\end{proof}

We can view $H=\Delta_d/\{\pm 1\}$ as a subgroup of ${\rm Aut}(X_1(N))$.
Since all elements $\sigma\in H$ fix a cusp, $H$ should be cyclic.
By Theorem \ref{central}, we have
 
\begin{lem}\label{genus}
If $g_1(N)-eg_{\D_d}(N)\geq e$, then the cusps $\sm x\\ y\esm$ with $d=\gcd(y,N)$ are Weierstrass points on $X_1(N)$.
\end{lem}

We need to find the relations of the functions involved with the genus formula in Theorem \ref{Gn} when $\D=\D_d$ and $\D=\{\pm1\}$.
First, we easily see that 
$\mu(N,\{\pm 1\})-e\mu(N, \D_d)=0.$
For $N>3$, there are no elliptic points on $X_1(N)$; hence $\nu_2(N,\{\pm1\})=\nu_3(N,\{\pm1\})=0$ for $N>3$.
Since $X_1(2)$ and $X_1(3)$ are of genera 0, we do not need to consider the two cases.
Clearly we have $\nu_2(N, \D_d)\geq 0$ and $\nu_3(N, \D_d)\geq 0$.
Finally, we have the following inequality:
\begin{lem}\label{lem:cusp2}
$$e\nu_\infty(N,\D_d)-\nu_\infty(N,\{\pm 1\})\geq (e-1)\frac{\varphi(d)\varphi(N/d)}{2}.$$
\end{lem} 
\begin{proof} 
From the proof of Theorem \ref{Gn} in \cite{JK}, we know that for each $c\mid N$, there are $\frac{\varphi(c)\varphi(\frac{N}{c})}{2}$ cusps on $X_1(N)$ lying above $\frac{\varphi(c)\varphi(\frac Nc)}{|\pi_c(\D)|}$ cusps on $X_{\D}(N)$.
Since the map $X_1(N)\to X_{\D_d}(N)$ is of degree $e$, the inequality
$$e\frac{\varphi(c)\varphi(\frac Nc)}{|\pi_c(\D_d)|}\geq \frac{\varphi(c)\varphi(\frac{N}{c})}{2}$$
always holds, and $|\pi_d(\D_d)|=2$.
Thus
\begin{align*}
e\nu_\infty(N,\D_d)-\nu_\infty(N,\{\pm 1\})&=e\sum_{c|N}\frac{\varphi(c)\varphi(\frac Nc
)}{|\pi_c(\D_d)|}-\sum_{c|N}\frac{\varphi(c)\varphi(\frac Nc
)}{2}\\
&\geq (e-1)\frac{\varphi(d)\varphi(N/d)}{2}
\end{align*}
\end{proof}

By Lemma \ref{lem:cusp2} together with Lemma \ref{genus} and Theorem \ref{Gn}, we have the following result:

\begin{lem}\label{cusp} If the inequality
$$\varphi(d)\varphi(N/d)\geq 8+\frac{4}{e-1}$$
holds, then the cusps $\sm x\\ y\esm$ with $d=\gcd(y,N)$ are Weierstrass points on $X_1(N)$.
\end{lem}

For a cusp $s=\sm x\\y\esm$ with $d=\gcd(y,N)$, $W_Ns=\sm -y/d\\xN/d\esm$; hence $W_N$ maps $s$ to a cusp $\sm x'\\y'\esm$ with $\gcd(y',N)=N/d$.
If $s$ is a Weierstrass point, then so is $W_Ns$.
Thus it suffices to consider $d$ such that $\varphi(d)\leq \varphi(N/d)$; hence if $\varphi(d)^2\geq 8+\frac{4}{e-1},$
then the cusps $\sm x\\ y\esm$ with $d=\gcd(y,N)$ are Weierstrass points on $X_1(N)$.
Therefore Theorem \ref{main} holds for $d\neq 2,3,4,6$.

First, consider the case when $d=2$.
Under the condition $g_1(N)\geq 2$, applying Lemma \ref{cusp}, we have that all cusps $\sm x\\ y\esm$ with $\gcd(y,N)=2$ of $X_1(N)$ are Weierstrass points except for $N=16,20,24,28,32,36,40,44,48,60$.
By Lemma \ref{genus}, we have the same conclusion for all remaining numbers except for $N=16,20$.

The author \cite{J} computed all Weierstrass points on the hyperelliptic curves $X_1(N)$, i.e. $N=13,16,18$, and he found that all irregular cusps of $X_1(16)$ are Weierstrass points.

Yang \cite{Y} computed generators for the function field of $X_1(N)$ by using the generalized Dedekind eta functions defined as follows:
$$E_r(\tau)=q^{NB(r/N)/2}\prod_{m=1}^\infty(1-q^{(m-1)N+r})(1-q^{mN-r}),$$
where $r$ is an integer not congruent to $0$ modulo $N$, $B(x)=x^2-x+1/6$ and $q=e^{2\pi i\tau}$.
By using the method explained in \cite{Y}, one can find that the two functions
\begin{equation}\label{fg}
f=\frac{E_2E_4^2E_6^2}{E_1^2E_8E_9^2}\hbox{ and } g=\frac{E_3E_4^2E_5E_6E_7}{E_1^2E_8^2E_9E_{10}}
\end{equation}
have poles of order 3 and 4 at the cusp $s=\sm1\\10\esm$, respectively and are regular elsewhere.
Thus the gap sequence of $s$ is $1,2,5$, and hence $s$ is a Weierstrass point of weight $2$ on $X_1(20)$.
There are 4 irregular cusps, namely $s$, $\sm3\\10\esm$, $\sm1\\2\esm$ and  $\sm 1\\6\esm$.
Since $W_{4}s=\sm3\\10\esm$, $W_{20}s=\sm1\\2\esm$ and $W_5s=\sm 1\\6\esm$, all irregular cusps of $X_1(20)$ are Weierstrass points of weight 2.

Second, consider the case when $d=3$.
Under the condition $g_1(N)\geq 2$, applying Lemma \ref{cusp}, we see that all cusps $\sm x\\ y\esm$ with $\gcd(y,N)=3$ of $X_1(N)$ are Weierstrass points except for $N=18,36$.
Lemma \ref{genus} takes care of $N=36$.
However, according to the computation in \cite[Theorem 3.4]{J}, all Weierstrass points of $X_1(18)$ are non-cusps.
Indeed, $X_1(18)$ is the only curve whose irregular cusps are not Weierstrass points among the curves $X_1(N)$ of $g_1(N)\geq 2$.

Third, consider the case when $d=4$.
Under the condition $g_1(N)\geq 2$, applying Lemma \ref{cusp}, we have that all cusps $\sm x\\ y\esm$ with $\gcd(y,N)=4$ of $X_1(N)$ are Weierstrass points except for $N=16,32,48$.
We treated $X_1(16)$ already above, and Lemma \ref{genus} confirms our claim for $N=32,48$.

Finally, consider the case when $d=6$.
Under the condition $g_1(N)\geq 2$, applying Lemma \ref{cusp}, we have that all cusps $\sm x\\ y\esm$ with $\gcd(y,N)=6$ of $X_1(N)$ are Weierstrass points except for $N=36,72$, for which we again use Lemma \ref{genus}.

Summarizing all results as above, we obtain Theorem \ref{main}.

\section{Weierstrass points of $X_0(N)$ at irregular cusps}

Throughout this section, we assume that $p$ is a prime and $N=p^2M$ for some integer $M$.
In this section, we give some results on when the irregular cusps of $X_0(p^2M)$ equivalent to $\sm 1\\p\esm$ are Weierstrass points.
First, we begin by a simple lemma.

\begin{lem}\cite{A}\label{rep}
$\Gamma_0(p^2M)$ is of index $p$ in $\Gamma_0(pM)$, and a set of right representatives for $\Gamma_0(p^2M)$ in $\Gamma_0(pM)$ is given by $\tau'=\tau/(kpM\tau+1)$ for $k=0,1,\dots,p-1$.
\end{lem}

First we consider the case when $p\mid M$.

\begin{lem}\label{tot} Suppose $p$ is a prime and $p|M$. Then the map $X_0(p^2M)\to X_0(pM)$ is of degree $p$ and it is totally ramified at each irregular cusp $\sm x\\ p\esm$ with $x=1,2,\dots,p-1$.
\end{lem}
\begin{proof} The result follows from the fact that the number of cusps on $X_0(p^2M)$ corresponding to the divisor $p$ is the same as for $X_0(pM)$ by Lemma \ref{lem:cusp} when $p|M$. Indeed, it is easy to check that all right representatives in Lemma \ref{rep} define the same cusp on $X_0(p^2M)$ for $\tau=\sm x\\ p\esm$.
\end{proof}

By Lemma \ref{tot} and the exact same proof of Theorem \ref{central}, we have the following result:
\begin{lem}\label{irr-wp} Suppose $p$ is a prime and $p|M$. If $g_0(p^2M)-pg_0(pM)\geq p$, then irregular cusps $\sm x\\ p\esm$ of $X_0(p^2M)$ with $x=1,2,\dots,p-1$ are Weierstrass points.
\end{lem}




\begin{rmk}
We note that the map $X_0(p^2M)\to X_0(pM)$ is not Galois for $p\geq 5$ and  with the Hurwitz formula Lemma \ref{irr-wp} is a special case of \cite[Lemma 1]{Li} (with the correction that $P$ is a total branching point).
\end{rmk}

The condition in Lemma \ref{irr-wp} is a sufficient condition for both the $0$-cusp and an irregular cusp $\sm x\\ p\esm$ to be Weierstrass points; hence we have the same results as \cite[Theorem 1]{A} for the case when $p|M$ except for $N=64$.

By \cite[Theorem 16]{B}, $S_p=\sm 1&\frac{1}{p}\\ 0& 1\esm$ is contained in $N(\Gamma_0(p^2M))$ for $p=2,3$; hence it defines an automorphism on $X_0(p^2M)$, and  it maps the $0$-cusp to $\sm 1\\p\esm$.
Thus the $0$-cusp is a Weierstrass point if and only if so is $\sm 1\\p\esm$ for $p=2,3$.

Since the $0$-cusp is a Weierstrass point on $X_0(64)$, $\sm 1\\2\esm$ is also a Weierstrass point.
Also since the $0$-cusp is not a Weierstrass point on $X_0(81)$, $\sm 1\\3\esm$ is also not.
Thus under the assumptions $g_0(p^2M)\geq 2$ and $p|M$, the irregular cusps $\sm x\\ p\esm$ are Weierstrass points if $p\geq 5$, $p=3$ with $M\neq 9$, or $p=2$ with $M\neq 2q, 4q$ ($q$ odd prime) by \cite[Theorem 1]{A}.

If $p=2$, by Lemma \ref{lem:cusp}, $\sm 1\\ 2\esm$ is a $\Q$-rational cusp; hence we can apply \cite[Theorem]{O3}, and we conclude that $\sm 1\\ 2\esm$ is not a Weierstrass point when $p=2$ with $M= 2q, 4q$ ($q$ odd prime). We note that $g_0(8)=g_0(16)=0$.

Thus have the following:
\begin{thm}\label{main2}
Suppose $p$ is a prime, $p\mid M$ and $g_0(p^2M)\geq 2$. Then the irregular cusps equivalent to $\sm 1\\p\esm$ are Weierstrass points of $X_0(p^2M)$  if and only if
$$p^2M\neq 81,8q,16q,$$
for some odd prime $q$. 
\end{thm}

In the cases $p=2$ or $p=3$, there is more that we can find out even if $p\nmid M$.
As stated above, the $0$-cusp is a Weierstrass point if and only if so is $\sm 1\\p\esm$ for $p=2,3$.
If $p=2$, we note again that $\sm 1\\2\esm$ is a $\Q$-rational cusp; hence we can apply \cite[Theorem]{O3} for this case.
From \cite[Theorem 5]{LN} and \cite[Theorem]{O3}, we have the following:
\begin{thm}\label{main3}
Suppose $M$ is odd and $g_0(4M)\geq 2$. Then the irregular cusps equivalent to $\sm 1\\2\esm$ are not Weierstass points when $M=p, 3p$ with an odd prime $p\neq 3$, and they are Weierstrass points of $X_0(4M)$, except possibly for
\begin{equation*}
M=pq
\end{equation*}
where $p$ and $q$ are distinct odd primes, neither is 3, and $q\equiv -1\pmod{4}$.
\end{thm}

For $p=3$, $\sm 1\\3\esm$ is not a $\Q$-rational cusp, we cannot apply \cite[Theorem]{O3} for this case.
Thus by \cite[Theorem 6]{LN}, we have the following:

\begin{thm}\label{main4}
Suppose $g_0(9M)\geq 2$ and $3\nmid M$. Then the irregular cusps equivalent to $\sm 1\\3\esm$ are Weierstrass points of $X_0(9M)$, except possibly for
\begin{equation*}
M=p, pq
\end{equation*}
where $p$ and $q$ are distinct primes, and $q\equiv -1\pmod{3}$.
\end{thm}

\vspace{0.5cm}

\begin{center}
{\bf Acknowledgment}
\end{center}
The author thanks to Yifan Yang for letting him know the two functions $f$ and $g$ in \eqref{fg}. Also the author is grateful to Andreas Schweizer for his useful comments.


\vspace{0.5cm}

\begin{thebibliography}{99}

\bibitem[A]{A} A.~O.~L.~Atkin, \emph{Weierstrass points at cusps of $\G_0(n)$}, Ann. Math., Second Series, \textbf{85} No.1 (Jan. 1967), 42--45.

\bibitem[B]{B} F.~Bars, \emph{The group structure of the normalizer of $\Gamma_0(N)$ after Atkin-Lehner}, Comm. Algebra \textbf{36} (2008), no. 6, 2160--2170.

\bibitem[C]{Ch} S.~Choi, \emph{A Weierstrass point of $\Gamma_1(4p)$}, J. Chungcheong Math. Soc. \textbf{21} (2008), no. 4, 467–-470.

\bibitem[FH]{FH} M.~Furumoto and Y.~Hasegawa, \emph{Hyperelliptic quotients of modular curves $X_0(N)$},  Tokyo J. Math. {\bf 22} (1999), no. 1, 105--125.

\bibitem[IJK]{IJK} B.-H.~Im, D.~Jeon and C.~H.~Kim, \emph{Weierstrass points on certain modular groups}, \textbf{160} (2016), 586-–602.


\bibitem[J]{J} D.~Jeon, \emph{Weierstrass points on hyperelliptic modular curves}, Commun. Korean Math. Soc. \textbf{30} (2015), No. 4, 379-–384

\bibitem[JK]{JK} D.~Jeon and C.~H.~Kim, {\em On the arithmetic of certain modular curves},  Acta Arith. \textbf{130} (2007) 181--193.



\bibitem[KK]{KK} C.~H.~Kim and J.~K.~Koo, \emph{The normalizer of $\G_1(N)$ in ${\rm PSL}_2(\R)$}, Comm. Algebra \textbf{28} (2000), 5303–-5310.

\bibitem[LN]{LN} J.~Lehner and M.~Newman, \emph{Weierstrass points of $\G_0(N)$}. Ann of Math. (2) \textbf{79} (1964), 360--368.

\bibitem[L]{L} J. Lewittes, \emph{Automorphisms of compact Riemann surfaces}. Amer. J. Math. \textbf{85} (1963), no. 4, 734--752.

\bibitem[Li]{Li} M. P.~Limonov,  \emph{On a generalization of Lewitte's theorem on Weierstrass points}. (Russian) Sibirsk. Mat. Zh. {\bf 55} (2014), no. 6, 1328--1333; translation in Sib. Math. J. {\bf 55} (2014), no. 6, 1084--1088.

\bibitem[O1]{O1} A.~P.~Ogg, {\em Rational points on certain elliptic modular curves}. Analytic number theory(Pro. Sympos. Pure Math., Vol XXIV, St. Louis Univ., St. Louis, Mo., 1972), pp. 221--231. Amer. Math. Soc., Providence, R.I., 1973.

\bibitem[O2]{O2} A.~P.~Ogg, \emph{Hyperelliptic modular curves}, Bull. Soc. Math. France \textbf{102} (1974), 449--462.

\bibitem[O3]{O3} A.~P.~Ogg, \emph{Weierstrass points of $X_0(N)$}, Illinois J. Math. \textbf{22} (1978), 31--35.


\bibitem[P]{P} H.~Petersson, \emph{Zwei Bemerkungen \"uber die Weierstrasspunkte der Kongruenzgruppen}, Arch. Math., \textbf{2} (1950), 246--250.

\bibitem[S]{S} B.~Sch\"oneberg, \emph{\"Uber die Weierstrasspunkte in den K\"orpern der elliptischen Modulfunktionen}. Abh. Math. Sem. Univ. Hamburg \textbf{17} (1951), 104--111.

\bibitem[Y]{Y} Y.~Yang, \emph{Defining equations of modular curves}, Adv. Math. \textbf{204} (2006), 481--508.

\end{thebibliography}
\end{document}